\documentclass[10pt]{article}
\usepackage{mathrsfs}
\usepackage{amsmath,amscd,amsthm,amsfonts}
\oddsidemargin=0.5cm\topmargin=-1cm \numberwithin{equation}{section}

\theoremstyle{plain}
\newtheorem{theorem}{Theorem}[section]
\newtheorem{corollary}[theorem]{Corollary}
\newtheorem{definition}[theorem]{Definition}

\newtheorem{proposition}[theorem]{Proposition}
\newtheorem{remark}[theorem]{Remark}

\newtheorem{example}[theorem]{Example}
\newtheorem*{Theorem 1}{Theorem 1}
\newtheorem*{Theorem 2}{Theorem 2}

\begin{document}
\vskip 5.1cm
\title{A note on two types of Lyapunov exponents and entropies for $\mathbb{Z}^k$-actions\footnotetext{
The author is supported by  NSFC(No:11371120). }}
 \author {Yujun Zhu\\
\small {School of Mathematical Sciences}\\
\small {Xiamen University, Xiamen, 361005, P. R. China}}
\date{}
\maketitle
\begin{center}
\begin{minipage}{130mm}
{\bf Abstract}:
In this paper, two types of Lyapunov exponents: random Lyapunov exponents and directional Lyapunov exponents, and the corresponding entropies: random entropy and directional entropy, are considered for smooth $\mathbb{Z}^k$-actions. The close relation among these quantities are investigated and the formulas of them are given via the Lyapunov exponents of the generators.

\end{minipage}
\end{center}

\section{Introduction and preliminaries}

Let $M$ be a $d$-dimensional closed (i.e., compact and boundaryless) Riemannian manifold, $f:M \longrightarrow M$ a diffeomorphism and $\mu$ an $f$-invariant Borel probability measure on $M$. There are various of ways of measuring the complexity of the $\mathbb{Z}$-action generated by iterating $f$.

A geometric way of measuring the complexity is to estimate the exponential rate at which nearby orbits are separated. These rates of divergence are given by the Lyapunov exponents of $f$.
By the Multiplicative Ergodic Theorem, there exists an invariant measurable set $\Gamma_f$ with full $\mu$ measure such that
for any $x\in \Gamma_f$ there exist a decomposition $T_xM=\bigoplus_{j=1}^{s(f,x)}E_j(x)$
into subspaces $E_j(x)$ of dimension $d_j(f,x)$ and numbers
$\lambda_1(f,x)<\cdots<\lambda_{s(f,x)}(f,x)$ which satisfy the following properties: $Df(x)E_j(x)=E_j(f(x))$  for each $1\le j\le s(f,x)$ and
$$
\lim_{n\longrightarrow \pm \infty}\frac{1}{n}\log\|Df^n(x)u\|=\lambda_j(f,x)
$$
for any nonzero vector $u\in E_j(x)$. The above numbers $\lambda_1(f,x),\cdots, \lambda_{s(f,x)}(f,x)$ are called the \emph{Lyapunov exponents} of $f$ at $x$, and the collection $\{(\lambda_j(f,x),d_j(f,x)): 1\le j\le s(f,x), x\in \Gamma_f\}$ is called the \emph{spectrum} of $f$.
Entropies, including measure-theoretic entropy $h_{\mu}(f)$ and topological entropy $h(f)$, are important invariants which measure the complexity of $f$ from measure-theoretic and topological points of view. The topological entropy measures the exponential growth rate of the number of orbits
distinguishable with limit precision, and the measure-theoretic
entropy describes the information creation rate of the time evolution.

The relations among Lyapunov exponents, measure-theoretic entropy $h_{\mu}(f)$ and topological entropy $h(f)$ have been well investigated. It is well known that $h_{\mu}(f)$ and $h(f)$ can be related by the variational principle $ h(f)=\sup\{h_\mu(f): \mu\;\text{is}\; f-\text{invariant}\}$.
When $f$ is $C^1$, there is an inequality relating $h_{\mu}(f)$ and the Lyapunov exponents,
\begin{equation}\label{Rue}
 h_{\mu}(f)\le\int_M \sum_{\lambda_j(f, x)>0}\lambda_j(f, x)d_j(f, x)d\mu(x).
\end{equation}
 It is called the \emph{Ruelle's inequality}. When $f$ is $C^{1+\alpha}, \alpha>0$, and $\mu$ is absolutely continuous with respect to the Lebesgue measure (or $\mu$ is a SRB measure, i.e., $\mu$ has the absolutely continuous conditional measures on the unstable manifolds), then the equality in the above inequality holds, i.e.,
\begin{equation}\label{Pes}
 h_{\mu}(f)=\int_M \sum_{\lambda_j(f, x)>0}\lambda_j(f, x)d_j(f, x)d\mu(x).
\end{equation}
 It is called the \emph{Pesin's formula}. Moreover, when there is no hypothesis of absolute continuity on $\mu$, we have a so-called the \emph{Dimension formula}
\begin{equation}\label{Dim}
 h_{\mu}(f)=\int_M \sum_{\lambda_j(f, x)>0}\lambda_j(f, x)\gamma_j(f, x)d\mu(x),
\end{equation}
 where $\gamma_j(f, x)$ is, roughly speaking, the dimension of $\mu$ in the direction of the subspace $E_j(x)$. For more detailed information we refer to \cite{Ruelle}, \cite{Pesin}, \cite{Ledrappier-Young} and \cite{Katok}.

For the higher rank
group actions, especially  $\mathbb{Z}^k$-actions, $k>1$, on $M$, i.e, the dynamical systems generated
by $k$ commuting diffeomorphisms, the dynamics become much more complicated (see \cite{Schmidt}).
Let $T:\mathbb{Z}^k\longrightarrow$ Diff$^r(M, M), r\geq 1$, be a $C^r$ $\mathbb{Z}^k$-action on $M$, where Diff$^r(M, M)$ is the space of $C^r$ diffeomorphisms equipped with the $C^r$-topology. We denote the collection of generators of $T$ by
\begin{equation}\label{generator}
\mathcal{G}=\{f_i=T(\vec{e}_i)=T^{\vec{e}_i}:1\leq i\leq k\},
\end{equation}
where $\vec{e}_i=(0,\cdots,\stackrel{(i)}{1},\cdots,0)$ is the standard $i$-th generator of $\mathbb{Z}^k$. Let $\mu$ be a \emph{$T$-invariant} Borel probability measure on $M$, i.e., $\mu$ is $f_i$-invariant for each $i$. Based on the need in the study of lattice statistical mechanics, Ruelle \cite{Ruelle1} introduced a version of entropy for
$\mathbb{Z}^k$-actions, $k\geq2$. Clearly, both of the topological entropy $h(T)$ and metric entropy $h_{\mu}(T)$ in Ruelle's version are zero because the topological entropy $h(f_i)$ and metric entropy $h_{\mu}(f_i)$  are both finite for each $i$ (see \cite{Schmidt}). Therefore this version of entropy is in a sense coarse to reveal the complexity of $T$.

In recent years, in order to investigate the diverse and intricate dynamics of $T$, various of entropy-type invariants to measure the complexity of $T$ in different levels and from different points of views are introduced and investigated, such as directional entropy (\cite{Milnor}, \cite{Park} and \cite{Boyle}), and \cite{Zhu}), Fried average entropy (\cite{Fried} and \cite{Katok1}), slow entropy (\cite{Katok2} and \cite{Katok1}) and Friedland entropy (\cite{Friedland}, \cite{Geller}, \cite{Einsiedler}), etc. Another type of entropy to describe the chaotic behavior of $T$ is its random entropy (or say, the entropy of the random $\mathbb{Z}^k$-actions induced by $T$). Roughly speaking, a random $\mathbb{Z}^k$-action $\varphi$ of $T$ is a system generated by the random compositions of the generators of $T$ according to some distribution $m$.

In this paper, we will mainly consider random entropy and directional entropy of $T$. Random entropy is the average entropy of the nonautonomous dynamical systems in $\varphi$, while directional entropy is the entropy of the sub-dynamical system along a direction. We will see that these two quantities indeed have internal relations. In section 2, we will consider the random Lyapunov exponents and random entropy of $T$ and formulate them respectively via the Lyapunov exponents of the generators (Theorem \ref{thm1} and Corollary \ref{coro}). In Section 3, we will apply these formulas to investigate the directional Lyapunov exponents and the measure-theoretic directional entropy for $T$ and give the corresponding formulas for them via the Lyapunov exponents of the generators (Theorem \ref{thm2}), respectively.

In the end of this section, we recall some fundamental properties of  $T$ from \cite{Hu}.

\begin{proposition}[\cite{Hu}]\label{prop}
There exists a measurable set $\Gamma\subset M$ with $f_i(\Gamma)=\Gamma$ for each $i$ (in this case we call $\Gamma$ is $T$-\emph{invariant}) and $\mu(\Gamma)=1$, such that for any $x\in \Gamma$, there exist a decomposition of tangent space into subspaces
\begin{equation}\label{splitting}
T_xM=\bigoplus_{j=1}^{s(x)} E_j(x)
\end{equation}
(where \emph{dim}$E_j(x)=d_j(x)$, $s(x)$ and $d_j(x)$ are all measurable and $\sum_{j=1}^{s(x)}d_j(x)=d$) and numbers  $\lambda_{i,j}(x), 1\le i\le k, 1\le j\le s(x)$, satisfying the following properties:

(1) for each $E_j(x)$, we have the following invariance properties
\begin{equation}\label{inv}
Df_i(x)E_j(x)=E_j(f_i(x))\;\;\text{and}\;\; \lambda_{i,j}(f_{i'}(x))=\lambda_{i,j}(x),
\end{equation}
where $1\le i,i'\le k$.
and

(2) for each $t_1,\cdots,t_k\in \mathbb{Z}$ and each
 $1\le j\le s(x)$,
\begin{equation}\label{Lya}
\lim_{n\longrightarrow \pm \infty}\frac{1}{n}\log\|D(f_1^{t_1}\circ\cdots \circ f_k^{t_k})^n(x)u\|=\sum_{r=1}^kt_i\lambda_{r,j}(x)
\end{equation}
for any $0\neq u\in E_j(x)$.
\end{proposition}

Note that for each $1\le i\le k$ and $x\in \Gamma$, the following two collections
 $$
 \{\lambda_j(f_i,x): 1\le j\le s(f_i,x)\}\text{ and }\{\lambda_{i,j}(x): 1\le j\le s(x)\}
 $$
coincide with each other. We call the collection
$$
\{(\lambda_{i,j}(x),d_j(x)): 1\le i\le k, 1\le j\le s(x), x\in \Gamma\}
$$
 the \emph{spectrum} of $T$.
 When the measure $\mu$  is ergodic with respect to $T$, i.e., $\mu$  is ergodic with respect to each $f_i$, then the above $\lambda_{i,j}(x),d_j(x)$ and $s(x)$ are all constant a.s., which are then denoted by $\lambda_{i,j},d_j$ and $s$ respectively.

\section{Random  Lyapunov exponents and random entropy}

Let $T$ be a $C^1$ $\mathbb{Z}^k$-action with the generators as in (\ref{generator}) and $\mu$ a $T$-invariant Borel probability measure on $M$. The following basic notions are derived from Kifer \cite{Kifer} and Liu \cite{Liu} in which the ergodic theory of general random dynamical systems are systematically investigated. Let
\begin{equation}\label{Omega}
\Omega=\mathcal{G}^{\mathbb{Z}}=\prod_{-\infty}^\infty \mathcal{G}
\end{equation}
be the infinite product of $\mathcal{G}$, endowed with the product topology and the product Borel $\sigma$-algebra $\mathcal{A}$,
and let $\sigma$ be the left shift operator on $\Omega$ which is defined by $(\sigma\omega)_n=\omega_{n+1}$ for $\omega=(\omega_n)\in \Omega$. Given $\omega=(\omega_n)\in \Omega$,
we write $\varphi_{\omega}=\omega_0$ and
$$
\varphi_{\omega}^n:=\left\{\begin{array}{ll}
\varphi_{\sigma^{n-1}\omega}\circ\cdots\circ \varphi_{\sigma\omega}\circ \varphi_{\omega} \quad &\;\;n>0\\
id \quad &\;\;n=0\\
\varphi_{\sigma^{n}\omega}^{-1}\circ\cdots\circ \varphi_{\sigma^{-2}\omega}^{-1}\circ \varphi_{\sigma^{-1}\omega}^{-1} \quad &\;\;n<0.
\end{array}\right.
$$
Clearly, each $\omega$ induces a nonautonomous dynamical system generated by the sequence of diffeomorphisms $\{\varphi_{\sigma^i\omega}\}_{i\in \mathbb{Z}}$.
Let  $m$ be a probability measure on $\mathcal{G}$. We can define a probability
measure $\mathbf{P}_{m}=m^{\mathbb{Z}}$ on $\Omega$ which is invariant and ergodic with respect
to $\sigma$. By the induced $C^r, r\geq 1,$ \emph{random $\mathbb{Z}^k$-action $\varphi$ over $(\Omega, \mathcal{A}, \mathbf{P}_{m}, \sigma)$} we  mean the  system generated by the randomly composed maps $\varphi_{\omega}^n, n\in \mathbb{Z}$. It is also called a $C^r $  \emph{i.i.d. (i.e., independent and identically distributed) random $\mathbb{Z}^k$-action}. We are interested in dynamical behaviors of these actions
for $\mathbf{P}_{m}$-a.e. $\omega$ or on the average on $\omega$.

There is a natural skew product transformation $\Phi: \Omega\times M\longrightarrow \Omega\times M$ over $(\Omega, \sigma)$ which is defined by
\begin{equation}\label{Phi}
\Phi(\omega, x)=(\sigma\omega, \varphi_{\omega}(x)).
\end{equation}

From the Multiplicative Ergodic Theorem for random dynamical systems (see, Theorem 3.2 in Chapter 1 of \cite{Liu}, for example) and Proposition \ref{prop}, there exists a Borel set $\Lambda\subset \Omega\times \Gamma$ with $\mathbf{P}_{m}\times \mu(\Lambda)=1, \Phi(\Lambda)\subset \Lambda$ such that for each $(\omega, x)\in \Lambda$, there exist a decomposition of tangent space into subspaces as (\ref{splitting}), i.e.,
$$
T_xM=\bigoplus_{j=1}^{s(x)} E_j(x)
$$
 and numbers  $\lambda_j(\omega, x),  1\le j\le s(x)$, satisfying that for $0\neq u\in E_j(x), 1\le j\le s(x)$,
\begin{equation}\label{randomLya}
\lim_{n\longrightarrow \pm \infty}\frac{1}{n}\log\|D\varphi_{\omega}^n(x)u\|=\lambda_j(\omega, x).
\end{equation}
We call $\lambda_j(\omega, x), 1\le j\le s(x)$, the \emph{$j$-th Lyapunov exponents} of $\varphi$, or the \emph{$j$-th random Lyapunov exponents} of $T$, at $(\omega, x)$. The collection $\{(\lambda_j(\omega, x),d_j(x)) : \omega\in\Omega, x\in \Gamma, 1\le j\le s(x)\}$ is called the \emph{spectrum} of $\varphi$. In fact, these Lyapunov exponents are non-random (even for a general $C^1$ i.i.d. random dynamical system $\psi$ which is defined by replacing ``$\mathcal{G}$" by ``Diff$^1(M)$ in (\ref{Omega})). This was proved in \cite{Kifer} by applying ergodic decomposition of stationary measures. In the following, we will apply Birkhoff's Ergodic Theorem to express  $\lambda_j(\omega, x)$ in $\lambda_{i,j}(x)$, and hence obtain the  non-randomness of $\lambda_j(\omega, x)$ for random $\mathbb{Z}^k$-actions.

\begin{theorem}\label{thm1}
Let $T$ be a $C^1$ $\mathbb{Z}^k$-action, $\mu$ a $T$-invariant Borel probability measure on $M$ and $\varphi$ a random $\mathbb{Z}^k$-action over $(\Omega, \mathcal{A}, \mathbf{P}_{m}, \sigma)$.
Then for $\mathbf{P}_{m}\times \mu$ a.e. $(\omega, x)$ and each $1\le j\le s(x)$, we have
\begin{equation}\label{exponent2}
\lambda_j(\omega, x)=\sum_{i=1}^k m_i \lambda_{i,j}(x),
\end{equation}
where $m_i=m(\{f_i\})$.
\end{theorem}

\begin{proof}

Let $l$ be a positive integer, $\widetilde{T}$ the $\mathbb{Z}^k$-action with the generators $f_i^l, 1\le i\le k$ and $\widetilde{\varphi}$ the corresponding random $\mathbb{Z}^k$-action over $(\widetilde{\Omega}, \widetilde{\mathcal{A}}, \widetilde{\mathbf{P}}_{m}, \widetilde{\sigma})$, where $\widetilde{\Omega}=\{f^l_i\}^{\mathbb{Z}}$ and $\widetilde{\mathcal{A}}, \widetilde{\mathbf{P}}_{m}, \widetilde{\sigma}$ are the counterparts of  $\mathcal{A}, \mathbf{P}_{m}, \sigma$, respectively. Please note that for each $\omega=(\omega_n)\in \Omega$ with $\omega_n=f_{i_n}\in \mathcal{G}$, there is a corresponding element $\widetilde{\omega}=(\widetilde{\omega}_n)\in \widetilde{\Omega}$ with $\widetilde{\omega}_n=f_{i_n}^l$, and vice versa. If we denote $\widetilde{\Phi}$ the counterpart of $\Phi$ and $\widetilde{\Lambda}=\{(\widetilde{\omega},x):(\omega,x)\in \Lambda\}$, then, clearly, $\widetilde{\mathbf{P}}_{m}\times \mu(\widetilde{\Lambda})=1, \widetilde{\Phi}(\widetilde{\Lambda})\subset \widetilde{\Lambda}$.

By (\ref{randomLya}) and (\ref{inv}), for any $(\omega, x)\in \Lambda, 1\le j\le s(x)$ and positive integer $l$,
$$
\lim_{n\longrightarrow \pm \infty}\frac{1}{n}\log\|D[\varphi_{\omega}^n]^l(x)u\|=l\lambda_j(\omega, x)
$$
 for any nonzero $u\in E_j(x)$. From the observation $\widetilde{\varphi}_{\widetilde{\omega}}^n=[\varphi_{\omega}^n]^l$, we get
\begin{equation}\label{power1}
\lim_{n\longrightarrow \pm \infty}\frac{1}{n}\log\|D\widetilde{\varphi}_{\widetilde{\omega}}^n(x)u\|=l\lambda_j(\omega, x)
\end{equation}
immediately. Hence, the $j$-th Lyapunov exponent of $\widetilde{\varphi}$ at $(\widetilde{\omega}, x)\in \widetilde{\Lambda}$ exist. If we denote it by  $\widetilde{\lambda}_j(\widetilde{\omega}, x), 1\le j\le s(x)$, then by (\ref{power1}),
\begin{equation}\label{power}
\widetilde{\lambda}_j(\widetilde{\omega}, x)=l\lambda_j(\omega, x).
\end{equation}

Take $\varepsilon>0$. Since for any $x\in \Gamma$ and each $1\le i\le k, 1\le j\le s(x)$, the following limit
$$
\lim_{n\longrightarrow \pm \infty}\frac{1}{n}\log\|Df^n_i(x)u\|=\lambda_{i,j}(x)
$$
is uniform with respect to $u\in E_j(x)\cap \mathbb{S}^{k-1}$, where $\mathbb{S}^{k-1}$ is the unit sphere in the tangent space $T_x M$, then by Egorov's Theorem there exists $\Gamma_1\subset\Gamma$ with $\mu(\Gamma_1)\ge 1-\frac{\varepsilon}{2}$ such that the above limit is uniform respect to $x\in \Gamma_1$ and $u\in E_j(x)\cap \mathbb{S}^{k-1}$. Therefore, for any $\delta>0$ there exist $N_1>0$ such that for any $l\ge N_1$, each $1\le i\le k, 1\le j\le s(x)$, $x\in \Gamma_1$ and $u\in E_j(x)\cap \mathbb{S}^{k-1}$,
\begin{equation}\label{power2}
l(\lambda_{i,j}(x)-\delta)\le\log\|Df_i^l(x)u\|\le l(\lambda_{i,j}(x)+\delta).
\end{equation}

Fix $l\ge N_1$. Let $\widetilde{\varphi}$ be the corresponding random $\mathbb{Z}^k$-action over $(\widetilde{\Omega}, \widetilde{\mathcal{A}}, \widetilde{\mathbf{P}}_{m}, \widetilde{\sigma})$ and $\widetilde{\Phi}$ be the corresponding skew product over $(\widetilde{\Omega}, \widetilde{\sigma})$ as in the beginning of the proof. Let
$$
\widetilde{\Lambda}_1=\{(\widetilde{\omega},x)\in \widetilde{\Lambda}:x\in \Gamma_1\}
$$ and let $\chi_{\widetilde{\Lambda}-\widetilde{\Lambda}_1}$ be the characteristic function of $\widetilde{\Lambda}-\widetilde{\Lambda}_1$.
By Birkhoff's Ergodic Theorem, for each $1\le i\le k$, the function
\begin{eqnarray*}
\widetilde{\chi}_i(\widetilde{\omega},x):&=&\lim_{n\longrightarrow +\infty}\frac{1}{n}\sum_{s=0}^{n-1}\chi_{\widetilde{\Lambda}-\widetilde{\Lambda}_1}(\widetilde{\Phi}^s(\widetilde{\omega},x))\\
&=&\lim_{n\longrightarrow +\infty}\frac{1}{n}\text{card}\{0\le s\le n-1: \widetilde{\Phi}^s(\widetilde{\omega},x)\in \widetilde{\Lambda}-\widetilde{\Lambda}_1\}
\end{eqnarray*}
is defined for $\widetilde{\mathbf{P}}_{m}\times \mu$ a.e. $(\widetilde{\omega}, x)$. Moreover,
\begin{eqnarray*}
\varepsilon &\geq&\widetilde{\mathbf{P}}_{m}\times \mu(\widetilde{\Lambda}-\widetilde{\Lambda}_1)\\
&=&\int_{\widetilde{\Lambda}} \chi_{\widetilde{\Lambda}-\widetilde{\Lambda}_1}d\widetilde{\mathbf{P}}_{m}\times \mu\\
&=&\int_{\widetilde{\Lambda}} \widetilde{\chi}_id\widetilde{\mathbf{P}}_{m}\times \mu\\
&\geq& \int_{\{(\widetilde{\omega},x)\in \widetilde{\Lambda}:\widetilde{\chi}_i(\widetilde{\omega},x)>\sqrt{\varepsilon}\}} \widetilde{\chi}_i\widetilde d{\mathbf{P}}_{m}\times \mu\\
&>& \sqrt{\varepsilon} \cdot \widetilde{\mathbf{P}}_{m}\times \mu\big(\{(\widetilde{\omega},x)\in \widetilde{\Lambda}:\widetilde{\chi}_i(\widetilde{\omega},x)>\sqrt{\varepsilon}\}\big),
\end{eqnarray*}
and hence
$$
\widetilde{\mathbf{P}}_{m}\times \mu\big(\{(\widetilde{\omega},x)\in \widetilde{\Lambda}:\widetilde{\chi}_i(\widetilde{\omega},x)\le\sqrt{\varepsilon}\}\big)\geq 1-\sqrt{\varepsilon}.
$$
By Egorov's Theorem, there exist $\widetilde{\Lambda}_2\subset\widetilde{\Lambda}$ with $\widetilde{\mathbf{P}}_{m}\times \mu(\widetilde{\Lambda}_2)>1-\frac{\varepsilon}{2}$ and $N_2>0$ such that  for each $1\le i\le k, (\widetilde{\omega},x)\in \widetilde{\Lambda}_2$ and any $n>N_2$,
\begin{equation}\label{frequency}
\text{card}\{0\le s\le n-1: \widetilde{\Phi}^s(\widetilde{\omega},x)\in \widetilde{\Lambda}-\widetilde{\Lambda}_1\}\le 2n\sqrt{\varepsilon}.
\end{equation}

 Given $(\widetilde{\omega}, x)\in \widetilde{\Lambda}$ and $1\leq j\leq s(x)$, we denote $\lambda_{\widetilde{\omega}, j}(x)=l\lambda_{i,j}(x)$ if $\widetilde{\omega}_0=f_i^l$. Then by (\ref{power2}) and (\ref{frequency}), for $(\widetilde{\omega},x)\in \widetilde{\Lambda}_2$, $u\in E_j(x)\cap \mathbb{S}^{k-1}$ and $n>N_2$,
\begin{eqnarray}\label{power3}
&&\sum_{r=0}^{n-1}\Big(\lambda_{\widetilde{\sigma}^r(\widetilde{\omega}), j}(x)-l\delta\Big)-2ln\sqrt{\varepsilon}\log(\|T\|)\notag\\
&\le&\log\|D\widetilde{\varphi}_{\widetilde{\omega}}^n(x)u\|\\
&\le& \sum_{r=0}^{n-1}\Big(\lambda_{\widetilde{\sigma}^r(\widetilde{\omega}), j}(x)+l\delta\Big)+2ln\sqrt{\varepsilon}\log(\|T\|),\notag
\end{eqnarray}
where $\|T\|=\max_{1\le i\le k}\|f_i\|$.

Now for any $x\in \Gamma$ which satisfies that $(\widetilde{\omega},x)\in \widetilde{\Lambda}$ for $\widetilde{\mathbf{P}}_{m}$-a.e. $\widetilde{\omega}$ and each $1\leq j\leq s(x)$, define a function $\eta^{(j)}_x:\widetilde{\Omega}\longrightarrow \mathbb{R}$ by
$$
\eta^{(j)}_x(\widetilde{\omega}):=\left\{\begin{array}{ll}
\lambda_{\widetilde{\omega}, j}(x) \quad &\;\;(\widetilde{\omega},x)\in \widetilde{\Lambda}_2\\
1 \quad &\;\;(\widetilde{\omega},x)\notin \widetilde{\Lambda}_2.
\end{array}\right.
$$
 Clearly, it is $\widetilde{\mathbf{P}}_{m}$-integrable.
Since $\widetilde{\mathbf{P}}_{m}$ is $\widetilde{\sigma}$-ergodic, by Birkhoff's Ergodic Theorem, we have that for $\widetilde{\mathbf{P}}_{m}$-almost all $\widetilde{\omega}\in \widetilde{\Omega}$,
\begin{equation}\label{power4}
\lim_{n\rightarrow\infty}\frac{1}{n}\sum_{r=0}^{n-1}\widetilde{\eta}_x(\widetilde{\sigma}^r
\widetilde{\omega})
=\int_{\widetilde{\Omega}}\eta^{(j)}_x(\widetilde{\omega})d\widetilde{\mathbf{P}}_m
(\widetilde{\omega}).
\end{equation}
Note that
\begin{equation}\label{power5}
l\cdot\big(\sum_{i=1}^km_i\lambda_{i,j}(x)-\frac{\varepsilon}{2}\log(\|T\|)\big)\le\int_{\widetilde{\Omega}}\eta^{(j)}_x(\widetilde{\omega})d\widetilde{\mathbf{P}}_m
(\widetilde{\omega})
\le l\cdot\big(\sum_{i=1}^km_i\lambda_{i,j}(x)+\frac{\varepsilon}{2}\log(\|T\|)\big).
\end{equation}
Combining (\ref{power1}), (\ref{power}), (\ref{power3}), (\ref{power4}) and (\ref{power5}), we have that for any $x\in \Gamma$ which satisfies that $(\omega,x)\in \Lambda$ for $\mathbf{P}_{m}$-a.e. $\omega$ and each $1\leq j\leq s(x)$
$$
\sum_{i=1}^k m_i (\lambda_{i,j}(x)-\delta)-(\frac{\varepsilon}{2}+2\sqrt{\varepsilon})\log(\|T\|)\le \lambda_j(\omega, x)\le \sum_{i=1}^k m_i  (\lambda_{i,j}(x)+\delta)+(\frac{\varepsilon}{2}+2\sqrt{\varepsilon})\log(\|T\|).
$$
Since $\varepsilon$ and $\delta$ can be taken arbitrarily small, we have that for $\mathbf{P}_{m}\times \mu$ a.e. $(\omega, x)$,
$$
\lambda_j(\omega, x)= \sum_{i=1}^k m_i  \lambda_{i,j}(x),
$$
i.e., the desired formula (\ref{exponent2}) holds.
\end{proof}

For a finite or countable Borel partition $\mathcal{P}$ of $M$,  the limit
\begin{equation}\label{hP}
h_{\mu}(\varphi,\mathcal{P}):=\lim_{n\rightarrow\infty}\frac{1}{n}\int_{\Omega}
H_{\mu}
\big(\bigvee_{i=0}^{n-1}(\varphi_{\omega}^i)^{-1}\mathcal{P}
\big)\;d\mathbf{P}_m(\omega),
\end{equation}
where $H_{\mu}(\mathcal{Q}):=-\sum\limits_{A\in\mathcal{Q}}\mu(A)\log\mu(A)$ for a finite or countable
partition $\mathcal{Q}$ of $M$, exists. The number
$$
h_{\mu}(\varphi):=\sup_{\mathcal{P}}h_{\mu}(\varphi,\mathcal{P}),
$$
where $\mathcal{P}$ ranges over all finite  or countable partitions of $M$, is called the
\emph{entropy} of $\varphi$ or the \emph{random entropy} of $T$. In fact, from \cite{Kifer2}, the
formula (\ref{hP}) remains true $\mathbf{P}_m$-a.s. without integrating against $\mathbf{P}_m$,
$$
h_{\mu}(\varphi,\mathcal{P})=\lim_{n\rightarrow\infty}\frac{1}{n}H_{\mu}
\big(\bigvee_{i=0}^{n-1}(\varphi_{\omega}^i)^{-1}\mathcal{P}
\big)
$$
for $\mathbf{P}_m$ a.e. $\omega\in \Omega$. Hence,
\begin{equation}\label{hP1}
h_{\mu}(\varphi)=\sup_{\mathcal{P}}\lim_{n\rightarrow\infty}\frac{1}{n}
H_{\mu}
\big(\bigvee_{i=0}^{n-1}(\varphi_{\omega}^i)^{-1}\mathcal{P}
\big)=h_{\mu}(\{\varphi_{\sigma^i\omega}\}_{i\in \mathbb{Z}_+})
\end{equation}
for $\mathbf{P}_m$ a.e. $\omega\in \Omega$.

Similar to (\ref{Rue}), (\ref{Pes}) and (\ref{Dim}) in the deterministic case, there are analogues of Ruelle's inequality, Pesin's formula and dimension formula for the random dynamical systems which reveal the relationship between random entropy and random Lyapunov exponents (see \cite{Liu} and \cite{Qian}), i.e., when $T$ is $C^1$,
\begin{equation}\label{inequality}
h_{\mu}(\varphi)\leq\int_{\Omega\times M}\sum_{\lambda_j(\omega, x)>0} d_j(x)\lambda_j(\omega, x)d\mathbf{P}_{m}\times \mu(\omega, x),
\end{equation}
and if $T$ is $C^2$ and $\mu$ is an SRB measure with respect to $\varphi$ (i.e., $\mu$ has the absolutely continuous conditional measures on the unstable manifolds of $\varphi$), then the equality in (\ref{inequality}) holds, moreover, when there is no hypothesis of absolute continuity on $\mu$, we have a so-called the Dimension formula,
$$
h_{\mu}(\varphi)=\int_{\Omega\times M} \sum_{\lambda_j(\omega, x)>0}\gamma_j(x)\lambda_j(\omega, x)d\mathbf{P}_{m}\times \mu(\omega, x),
$$
where $\gamma_j(x)$ is the dimension of $\mu$ in the direction of the subspace $E_j(x)$. Hence by Theorem \ref{thm1}, we have the following results immediately.

\begin{corollary}\label{coro}
Under the assumptions of Theorem 2.1, we have the following results.

(1) (Ruelle's inequality)
When $T$ is $C^1$, we have
\begin{equation}\label{Ruelle's inequality2}
h_{\mu}(\varphi)\leq\int_M\max_{J\subset \{1,\cdots,s(x)\}}\sum_{j\in J}\sum_{i=1}^k m_i d_j(x)\lambda_{i,j}(x)d\mu(x).
\end{equation}

(2)  (Pesin's formula)
If $T$ is $C^2$ and $\mu$ is an SRB measure with respect to $\varphi$, then the equality in (\ref{Ruelle's inequality2}) holds, i.e.,
\begin{equation}\label{Pesin's formula2}
h_{\mu}(\varphi)=\int_M\max_{J\subset \{1,\cdots,s(x)\}}\sum_{j\in J}\sum_{i=1}^k m_i d_j(x)\lambda_{i,j}(x)d\mu(x).
\end{equation}

(3)  (Dimension formula)
If $T$ is $C^2$ and there is no additional assumption of absolute continuity on $\mu$, then
\begin{equation}\label{Dimension formula2}
h_{\mu}(\varphi)=\int_M\max_{J\subset \{1,\cdots,s(x)\}}\sum_{j\in J}\sum_{i=1}^k m_i \gamma_j(x)\lambda_{i,j}(x)d\mu(x).
\end{equation}

\end{corollary}

\begin{remark}
In \cite{Zhu}, (\ref{Ruelle's inequality2}) and (\ref{Pesin's formula2}) in Corollary \ref{coro} are also obtained (under stronger assumptions) by applying the adapted techniques of Hu (\cite{Hu}) and Ruelle (\cite{Ruelle}) and M$\tilde{a}$n$\acute{e}$ (\cite{Mane}) to the random case. The proof in \cite{Zhu} is in a sense self-contained and elaborated. But now after getting the formula (\ref{exponent2}),  we can directly obtain (\ref{Ruelle's inequality2}) and (\ref{Pesin's formula2}) by substitution. Furthermore, we can see an application of the ``random" entropy formula (\ref{Pesin's formula2}) to the ``deterministic" Friedland's entropy of $T$ in \cite{Zhu}. 
\end{remark}

\section{Directional Lyapunov exponents and directional entropy}

Subdynamics, especially expansive subdynamics, of a $\mathbb{Z}^k$-action are systematically investigated by Boyle and Lind \cite{Boyle}.  In particular, when consider one dimensional subdynamics, the directional entropy is studied. Roughly speaking, the directional entropy is a quantity to measure the complexity of the  sub-dynamical system along a direction. Directional entropy of $\mathbb{Z}^k$-actions was introduced by Milnor \cite{Milnor} to investigate the dynamics of the cellular automata, and then has been studied extensively (see for example \cite{Boyle}, \cite{Park} and \cite{Robinson}, etc). In \cite{Boyle}, it was shown that the directional entropies, including measure theoretic and topological versions, are continuous within an expansive component using the creative  techniques ``coding" and ``shading". However, it was not known how the directional entropy changes as the direction changes from one expansive component to another except for some special cases (see \cite{Park} and Remark 8.6 of \cite{Einsiedler}). As it was mentioned in the last section of \cite{Boyle}, ``\emph{continuity properties of directional entropy are still obscure}".

Let $T$ be a $C^1$ $\mathbb{Z}^k$-action with the generators as in (\ref{generator}) and $\mu$ a $T$-invariant Borel probability measure on $M$. In this section, we will introduce the directional Lyapunov exponents and express them in the Lyapunov exponents of the generators of $T$. Moreover, if $T$ is $C^2$, a formula of metric directional entropy for $T$ via the directional Lyapunov exponents is given, and hence the continuity of the directional entropy for $T$ is obtained.

Denote $\mathbb{S}^{k-1}$ the unit sphere in $\mathbb{R}^k$.
Given $\vec{v}=(v_1,\cdots,v_k)\in \mathbb{S}^{k-1}$. Let $L_{\vec{v}}$ be the one-dimensional subspace of  $\mathbb{R}^k$ which contains $\vec{v}$, i.e., the line on which $\vec{v}$ lies.
Let $\mathbb{Z}_+$ be the set of nonnegative integers.
Define a sequence of $\{\vec{m}_n\}_{n\in \mathbb{Z}_+}\subset
\mathbb{Z}^k$ as follows: choose $\vec{m}_n$ to be any integer vector in
$$
\{\vec{m}:|\vec{m}-n\vec{v}|=\displaystyle\min_{\vec{l}\in
\mathbb{Z}^k}|\vec{l}-n\vec{v}|\}
$$
with the smallest norm.
Obviously, for any $\vec{m}_n=(m_{n,1},\cdots,m_{n,k})$,
$\vec{m}_{n'}=(m_{n',1},\cdots,m_{n',k})$ and $1\le i\le k$,
$$
n<n'\Longrightarrow\left\{\begin{array}{ll}
m_{n,i}\leq m_{n',i}, \quad &\;\;v_i\ge 0\\
m_{n,i}\ge m_{n',i}, \quad &\;\;v_i\le 0.
\end{array}\right.
$$
Let $g_0=T^{\vec{m}_1},
g_n=T^{\vec{m}_n-\vec{m}_{n-1}}$ for $n\ge1$. Note that $g_n=$id when $\vec{m}_n=\vec{m}_{n-1}$. We call
$g^{\vec{v}}_{0,\infty}=\{g_n\}_{n\in \mathbb{Z}_+}$ is \emph{a
nonautonomous dynamical system along} $L_{\vec{v}}$ (in the direction of $\vec{v}$). Clearly, the sequence of $\{\vec{m}_n\}_{n\in \mathbb{Z}_+}$, and hence the nonautonomous dynamical system along $L_{\vec{v}}$,  may not be unique. However, we can see that for any $\{\vec{m}_n\}$, we have
\begin{equation}\label{limit}
\lim_{n\longrightarrow \infty}\Big(\frac{m_{n,1}}{\sum_{i=1}^k|m_{n,i}|},\cdots,\frac{m_{n,k}}{\sum_{i=1}^k|m_{n,i}|}\Big)=(v_1,\cdots,v_k).
\end{equation}

\begin{definition}\label{directional Lyapunov exponent}
Let $\vec{v}=(v_1,\cdots,v_k)\in \mathbb{S}^{k-1}$ and $g^{\vec{v}}_{0,\infty}=\{g_n\}_{n\in \mathbb{Z}_+}$ be a
nonautonomous dynamical system along $L_{\vec{v}}$. Then for any $x\in \Gamma$ and each $1\le j\le s(x)$, the limit
\begin{equation}\label{exponent}
\lim_{n\longrightarrow  \infty}\frac{1}{n}\log\|D(g_{n-1}\circ\cdots\circ g_0)(x)u\|,  u\in E_j(x)
\end{equation}
(in fact we can see in the proof of Theorem \ref{thm2} it exists and is independent of $u\in E_j(x)$ and the choice of $g^{\vec{v}}_{0,\infty}$) is called \emph{the $j$-th directional Lyapunov exponent of $T$ at $x$} in the direction of $\vec{v}$, and it is denoted
by $\lambda_j^{\vec{v}}(x)$.
\end{definition}

\begin{remark}
In Kalinin, Katok and Hertz's paper \cite{Kalinin} which investigates the nonuniform measure rigidity for $\mathbb{Z}^k$-actions, Lyapunov exponents for a $\mathbb{Z}^k$ action on $M$ are extended to a $\mathbb{R}^k$ action on the suspension manifold $S$, which is a bundle over $\mathbb{T}^k$ with fiber $M$. We can see that for any $\vec{v}=(v_1,\cdots,v_k)\in \mathbb{S}^{k-1}$, the Lyapunov exponents $\chi_j(\vec{v})$ in Proposition 2.1 of \cite{Kalinin} is just the direction Lyapunov exponents $\lambda_j^{\vec{v}}$ in our case.
\end{remark}

For any $\vec{v}=(v_1,\cdots,v_k)\in \mathbb{S}^{k-1}$, the directional entropy in $\vec{v}$ is defined as follows (see \cite{Boyle}, \cite{Milnor} and \cite{Park}).  For $t,r>0$, define
$$
L_{\vec{v}}^t(r)=\{\vec{u}\in\mathbb{R}_+^k:|\pi_{L_{\vec{v}}}\vec{u}|\leq r\mbox{ and
}|\pi_{L_{\vec{v}}^\perp}\vec{u}|\leq t\},
$$
where  $L_{\vec{v}}^\perp$ is the orthogonal complement of $L_{\vec{v}}$ and $\pi_{L_{\vec{v}}}$ (resp. $\pi_{L_{\vec{v}}^\perp}$) is the orthogonal
projection to $L_{\vec{v}}$ (resp. $L_{\vec{v}}^\perp$) along $L_{\vec{v}}^{\perp}$ (resp. $L_{\vec{v}}$). For a finite or countable Borel partition $\mathcal{P}$ of $M$, let
$$
h_{\mu}^{\vec{v}}(T, \mathcal{P})=\lim_{t\longrightarrow \infty}\limsup_{r\longrightarrow \infty}\frac{1}{r}H_{\mu}
\big(\bigvee_{\vec{n}\in L_{\vec{v}}^t(r)}T^{-\vec{n}}\mathcal{P}\big).
$$
The quantity
$$
h_{\mu}^{\vec{v}}(T)=\sup_{\mathcal{P}}h_{\mu}^{\vec{v}}(T, \mathcal{P}),
$$
where $\mathcal{P}$ ranges over all finite  or countable partitions of $M$, is called the \emph{directional entropy of $T$ in  $\vec{v}$}. There is an equivalent definition of  $h_{\mu}^{\vec{v}}(T)$ in \cite{Zhang} via the entropy of the  corresponding nonautonomous dynamical system. Precisely, let $g^{\vec{v}}_{0,\infty}=\{g_i\}_{i\in \mathbb{Z}_+}$ be any nonautonomous dynamical system along $L_{\vec{v}}$, then
\begin{equation}\label{entropy}
h_{\mu}^{\vec{v}}(T)=h_{\mu}(g^{\vec{v}}_{0,\infty})=\sup_{\mathcal{P}}h_{\mu}(g^{\vec{v}}_{0,\infty},\mathcal{P}),
\end{equation}
where
$$
h_{\mu}(g^{\vec{v}}_{0,\infty},\mathcal{P})=\limsup\limits_{n\rightarrow \infty}
 \frac{1}{n}H_{\mu}\Big(\bigvee_{i=0}^{n-1}(g_i\circ\cdots  \circ g_0)^{-1}\mathcal{P}\Big).
$$
For more information of entropy for nonautonomous dynamical system, we refer to \cite{Kolyada}, \cite{Zhu1} and \cite{Kawan}, etc.

\begin{theorem}\label{thm2}
Let $T$ be a $C^1$ $\mathbb{Z}^k$-action and $\mu$ a $T$-invariant Borel probability measure on $M$. Then for $\vec{v}=(v_1,\cdots,v_k)\in \mathbb{S}^{k-1}$, the following results hold.

(1)  For $\mu$ a.e. $x$ and each $1\le j\le s(x)$,
\begin{equation}\label{exponent1}
\lambda_j^{\vec{v}}(x)=\sum_{i=1}^{k}v_i\lambda_{i,j}(x).
\end{equation}

(2) (Ruelle's inequality)
\begin{equation}\label{Ruelle's inequality1}
h_{\mu}^{\vec{v}}(T)\le\int_M \max_{J\subset \{1,\cdots,s(x)\}}\sum_{j\in J}\sum_{i=1}^k v_i d_j(x)\lambda_{i,j}(x)d\mu(x).
\end{equation}

(3)  (Pesin's formula)
If $T$ is $C^2$ and $\mu$ is an SRB measure with respect to $\varphi$, then we have
\begin{equation}\label{Pesin's formula1}
h_{\mu}^{\vec{v}}(T)=\int_M \max_{J\subset \{1,\cdots,s(x)\}}\sum_{j\in J}\sum_{i=1}^k v_i d_j(x)\lambda_{i,j}(x)d\mu(x).
\end{equation}
Moreover, if $\mu$ is a general invariant measure, then we have the following dimension formula
\begin{equation}\label{Dimension formula1}
h_{\mu}^{\vec{v}}(T)=\int_M \max_{J\subset \{1,\cdots,s(x)\}}\sum_{j\in J}\sum_{i=1}^k v_i \gamma_j(x)\lambda_{i,j}(x)d\mu(x).
\end{equation}
And hence when $T$ is $C^2$, $h_{\mu}^{\vec{v}}(T)$ is continuous in $\vec{v}$.
\end{theorem}

\begin{proof}

Given $\vec{v}=(v_1,\cdots,v_k)\in \mathbb{S}^{k-1}$.

\emph{Step 1}. For simplicity, we first assume that $\vec{v}$ lies in the first octant, i.e.,  $v_i\ge 0$ for each $1\le i\le k$.

Let $m$ be the probability measure on $\mathcal{G}$ with
\begin{equation}\label{mi}
m_i=m(f_i)=\frac{v_i}{\sum_{i=1}^kv_i}, 1\le i\le k.
\end{equation}
We then define a random $\mathbb{Z}^k$-action $\varphi_{\vec{v}}$ over $(\Omega, \mathcal{A}, \mathbf{P}_{m}, \sigma)$.

Let $g^{\vec{v}}_{0,\infty}=\{g_i\}_{i\in \mathbb{Z}_+}$ be a  nonautonomous dynamical system along $L_{\vec{v}}$. We can see that it corresponds to a typical element $\omega^*\in\Omega$. Let $\{(\varphi_{\vec{v}})_{\sigma^i\omega^*}\}_{i\in \mathbb{Z}_+}$ be the one-sided nonautonomous dynamical system of $\omega^*$. Please note that for large $n$, the composition of the first $n$ elements in $\{g_i\}_{i\in \mathbb{Z}_+}$ approximately corresponds to the composition of the first $\lceil n\cdot \sum_{i=1}^kv_i\rceil$ elements in $\{(\varphi_{\vec{v}})_{\sigma^i\omega^*}\}_{i\in \mathbb{Z}_+}$.

(1) From the above observation, we have, for any $x\in \Gamma$, each $1\le j\le s(x)$ and $u\in E_j(x)$,
\begin{eqnarray*}
&&\lim_{n\longrightarrow  \infty}\frac{1}{n}\log\|D(g_{n-1}\circ\cdots\circ g_0)(x)u\|\\
&=&\lim_{n\longrightarrow  \infty}\frac{\lceil n\cdot \sum_{i=1}^kv_i\rceil}{n}\cdot\frac{1}{\lceil n\cdot \sum_{i=1}^kv_i\rceil}\log\|D(\varphi_{\vec{v}})_{\omega^*}^{\lceil n\cdot \sum_{i=1}^kv_i\rceil}(x)u\|\\
&=&\lim_{n\longrightarrow  \infty}\frac{\lceil n\cdot \sum_{i=1}^kv_i\rceil}{n}\cdot\sum_{i=1}^k m_i \lambda_{i,j}(x)\;\;\;\;(\text{by }  (\ref{randomLya}) \text{ and } (\ref{exponent2}))\\
&=&\sum_{i=1}^k v_i \lambda_{i,j}(x).\;\;\;(\text{by }  (\ref{mi}))
\end{eqnarray*}
Therefore, the limit in (\ref{exponent}) exists and is independent of $u\in E_j(x)$ and the choice of $g^{\vec{v}}_{0,\infty}$. So the $j$-th directional Lyapunov exponent $\lambda_j^{\vec{v}}(x)$ of $T$ at $x$ in the direction of $\vec{v}$ equals $\sum_{i=1}^k v_i \lambda_{i,j}(x)$, i.e., (\ref{exponent1}) holds.

(2) Similarly, since $T$ is $C^1$  and $\mu$ is $T$-invariant,
\begin{eqnarray*}
h_{\mu}^{\vec{v}}(T)&=&h_{\mu}(g^{\vec{v}}_{0,\infty})\\
&=&\sum_{i=1}^kv_i\cdot h_{\mu}(\{(\varphi_{\vec{v}})_{\sigma^i\omega^*}\}_{i\in \mathbb{Z}_+})\\
&\le&\sum_{i=1}^kv_i\cdot\int_M\max_{J\subset \{1,\cdots,s(x)\}}\sum_{j\in J}\sum_{i=1}^k m_i d_j(x)\lambda_{i,j}(x)d\mu(x)\;\;\;(\text{by }  (\ref{hP1}) \text{ and } (\ref{Pesin's formula2}))\\
&=&\int_M\max_{J\subset \{1,\cdots,s(x)\}}\sum_{j\in J}\sum_{i=1}^k v_i d_j(x)\lambda_{i,j}(x)d\mu(x),\;\;\;(\text{by }  (\ref{mi}))
\end{eqnarray*}
i.e., the inequality (\ref{Ruelle's inequality1}) holds.

(3) If $T$ is $C^2$, then apply (\ref{Pesin's formula2}) and (\ref{Dimension formula2}) to the discussion in (2) we have the formula (\ref{Pesin's formula1}) and (\ref{Dimension formula1}) immediately.  And hence $h_{\mu}^{\vec{v}}(T)$ is continuous in $\vec{v}$.

\emph{Step 2}.  Now we consider any $\vec{v}=(v_1,\cdots,v_k)\in \mathbb{S}^{k-1}$ which does not lie in the first octant. Define a new vector  $\vec{v'}=(v'_1,\cdots,v'_k)\in \mathbb{S}^{k-1}$ with $v'_i=|v_i|, 1\le i\le k$.
 Clearly  $\vec{v'}$ lies in the first octant.

 Define a new $\mathbb{Z}^k$-action $T'$ with the collection of generators
\begin{equation}\label{generator1}
\mathcal{G}'=\{f'_i :1\leq i\leq k\},
\end{equation}
in which
$$
f'_i=\left\{\begin{array}{ll}
f_i, \quad &\;\;v_i\ge 0\\
f_i^{-1}, \quad &\;\;v_i< 0.
\end{array}\right.
$$
  Let $m'$ be the probability measure on $\mathcal{G}'$ with
$$
m'_i=m'(f'_i)=\frac{v'_i}{\sum_{i=1}^kv'_i}, 1\le i\le k.
 $$
We then define a corresponding random $\mathbb{Z}^k$-action $\varphi'_{\vec{v'}}$ over $(\Omega', \mathcal{A}', \mathbf{P}'_{m'}, \sigma')$.
Let $\{(\lambda'_{i,j}(x),d_j(x)) : x\in \Gamma, 1\le i\le k, 1\le j\le s(x)\}$ be the spectrum of $T'$ and $\{(\lambda'_j(\omega', x),d_j(x)) : \omega'\in \Omega', x\in \Gamma, 1\le j\le s(x)\}$ the spectrum of $\varphi'_{\vec{v'}}$.

By step 1, for the new system of $T'$, $\varphi'_{\vec{v'}}$ and $\vec{v'}$, the the corresponding results in (1)-(3) hold. Please note that for any $x\in \Gamma, 1\le i\le k, 1\le j\le s(x)$,
$$
v_i \lambda_{i,j}(x)=v'_i \lambda'_{i,j}(x),
$$
and hence the desired results in (1)-(3) hold immediately.

This completes the proof of Theorem \ref{thm2}.
\end{proof}

Recall that a line $L$ through the origin is called \emph{expansive} for $T$ if  there are $\varepsilon>0$ and $t>0$ such that
$$
\sup\big\{\rho(T^{\vec{n}}(x),T^{\vec{n}}(y)):\vec{n}\in \{\vec{u}\in \mathbb{R}^k: \text{dist}(\vec{u},L)\le t\}\big\}\le\varepsilon
$$
implies $x=y$, where $\rho$ is the metric on $M$. Let $\mathbb{G}_1$ be the Grassmann manifold of all 1-dimensional subspaces of $\mathbb{R}^k$ and denote $\mathbb{E}_1(T)=\{L\in \mathbb{G}_1:L \text{ is expansive for } T\}$. From \cite{Boyle}, $\mathbb{E}_1(T)$ is open and both of the metric directional entropy $h_{\mu}^{\vec{v}}(T)$ and the topological directional entropy $h^{\vec{v}}(T)$ are continuous at $\vec{v}$ with $L_{\vec{v}}\in \mathbb{E}_1(T)$.

When $T$ is $C^2$ and $\mu$ is absolutely continuous with respect to the Lebesgue measure, we can obtain the formula (\ref{Pesin's formula1}) on expansive components (especially on \emph{Weyl chambers} for the particular ${\mathbb{Z}}^k$-actions in \cite{Kalinin}) by a ``deterministic" method mentioned in \cite{Boyle} as follows.

When $\vec{v}=(v_1,\cdots,v_k)\in \mathbb{S}^{k-1}$ is \emph{rational}, i.e., the components $v_i, 1\leq i\leq
k$, of $\vec{v}$ are rationally related, we take a minimal integer $n$  such that $n\vec{v}$ is an integer vector. Let $g^{\vec{v}}_{0,\infty}=\{g_i\}_{i\in \mathbb{Z}_+}$ be a nonautonomous dynamical system in the direction of $\vec{v}$. By Theorem 4.7 of \cite{Zhang},
\begin{equation}\label{hmu1}
h_{\mu}^{\vec{v}}(T)=\frac{1}{n}h_{\mu}(T^{n\vec{v}}).
\end{equation}
Clearly, $T^{n\vec{v}}=g_n\circ\cdots\circ g_1=f_1^{nv_1}\circ\cdots\circ f_k^{nv_k}$, and hence, for any $x\in \Gamma$, $1\le j\le s(x)$, the $j$-th Lyapunov exponent of $T^{n\vec{v}}$ at $x$ is $\sum_{i=1}^knv_i\lambda_{i,j}(x)=n\lambda_j^{\vec{v}}(x)$. By this observation we get  (\ref{Pesin's formula1}) immediately.

When $\vec{v}=(v_1,\cdots,v_k)\in \mathbb{S}^{k-1}$ is irrational with $L_{\vec{v}}$ lying in an expansive component $\mathcal{C}$. Since $\mathcal{C}$ is open $h_{\mu}^{\vec{v}}(T)$ is continuous on $\mathcal{C}$, we can take a sequence of rational vectors $\{\vec{v}_n=(v_{n,1},\cdots,v_{n,k})\}_{n\in \mathbb{Z}_+}\subset \mathbb{S}^{k-1}$ such that $L_{\vec{v}_n}\in \mathcal{C}$ and $\vec{v}_n\longrightarrow \vec{v}$ as $n\longrightarrow \infty$. Therefore, we have
\begin{eqnarray}
h_{\mu}^{\vec{v}}(T)&=& \lim_{n\longrightarrow \infty}h_{\mu}^{\vec{v}_n}(T)\notag\\
&=&\lim_{n\longrightarrow \infty}\int_M \max_{J\subset \{1,\cdots,s(x)\}}\sum_{j\in J}\sum_{i=1}^k v_{n,i} d_j(x)\lambda_{i,j}(x)d\mu(x)\notag\\
&=&\int_M \max_{J\subset \{1,\cdots,s(x)\}}\sum_{j\in J}\sum_{i=1}^k v_i d_j(x)\lambda_{i,j}(x)d\mu(x),\notag
\end{eqnarray}
i.e.,  (\ref{Pesin's formula1}) holds.

However, when $\vec{v}=(v_1,\cdots,v_k)\in \mathbb{S}^{k-1}$ is irrational with $L_{\vec{v}}$ lying on the boundary of an expansive component, the above method does not work. This is the reason why we resort to the formulas in random dynamical systems to obtain the formulas in deterministic dynamical systems as in Theorem \ref{thm2}.

\begin{example}
Let $T$ be a linear ${\mathbb{Z}}^k$-action on the torus
$\mathbb{T}^d$ with the generators $\{f_i\}_{i=1}^k$ which are induced by integer  matrices $\{A_i\}_{i=1}^k$. Let $\{(\lambda_{i,j},d_j) :  1\le i\le k, 1\le j\le s\}$ be the spectrum of $T$ (in fact each $\lambda_{i,j}$ is the logarithm of the absolute value of an eigenvalue of $A_i$). It is easy to see that the Lebesgue measure $m_{Leb}$ is a measure with maximal directional entropy in any direction, i.e., $h^{\vec{v}}(T)=h_{m_{Leb}}^{\vec{v}}(T)$ for any $\vec{v}=(v_1,\cdots,v_k)\in \mathbb{S}^{k-1}$. Therefore, by Theorem \ref{thm2} we have that for linear ${\mathbb{Z}}^k$-actions on $\mathbb{T}^d$ the topological directional entropy $h^{\vec{v}}(T)$ is continuous in $\vec{v}$.
\end{example}

In the end of this section, we point out that the directional entropy $h_{\mu}^{\vec{v}}(T)$  is defined for any vector $\vec{v}\in {\mathbb{R}}^k$  in \cite{Boyle}. Since $h_{\mu}^{\vec{0}}(T)=0$ and the formula $h_{\mu}^{\vec{v}}(T)=|\vec{v}|\cdot h_{\mu}^{\frac{\vec{v}}{|\vec{v}|}}(T)$ holds for any nonzero vector $\vec{v}\in {\mathbb{R}}^k$, we can get the same results as that in this section for the directional entropy $h_{\mu}^{\vec{v}}(T)$ of any vector $\vec{v}\in {\mathbb{R}}^k$.

\end{document}